\documentclass{article}
\usepackage[margin=1.5in]{geometry}
\usepackage[british]{babel}
\usepackage{amsthm}
\usepackage{mathtools}
\usepackage{amssymb, enumerate}
\usepackage{latexsym}
\usepackage{url}
\usepackage[all,cmtip]{xy}
\usepackage{amsmath,amssymb}
\usepackage{tikz-cd}
\usepackage{faktor}

\newcounter{itemcounter}
\numberwithin{itemcounter}{section}

\newtheorem{assumption}[itemcounter]{Assumption}
\newtheorem{thm}[itemcounter]{Theorem}
\newtheorem{lem}[itemcounter]{Lemma}

\newtheorem{prop}[itemcounter]{Proposition}

\newtheorem{rem}[itemcounter]{Remark}

\newtheorem{property}[itemcounter]{Property}
\newtheorem*{thm*}{Theorem}
\newtheorem*{con*}{Conjecture}
\newtheorem*{cor*}{Corollary}
\newtheorem*{ack*}{Acknowledgements}

\newcommand{\Inf}{\mathop{\rm Inf}\nolimits}

\newcommand{\Irr}{\mathop{\rm Irr}\nolimits}

\newcommand{\IBr}{\mathop{\rm IBr}\nolimits}

\newcommand{\Hom}{\mathop{\rm Hom}\nolimits}

\newcommand{\Aut}{\mathop{\rm Aut}\nolimits}
\newcommand{\Ext}{\mathop{\rm Ext}\nolimits}
\newcommand{\Ann}{\mathop{\rm Ann}\nolimits}
\newcommand{\Tor}{\mathop{\rm Tor}\nolimits}

\newcommand{\Pic}{\mathop{\rm Pic}\nolimits}

\newcommand{\nth}{\mathop{\rm th}\nolimits}

\newcommand{\cT} {\mathcal{T}}
\newcommand{\cL} {\mathcal{L}}
\newcommand{\cE} {\mathcal{E}}
\newcommand{\cO} {\mathcal{O}}

\newcommand{\NN} {\mathbb{N}}
\newcommand{\ZZ} {\mathbb{Z}}

\title{Picard groups for blocks with normal defect groups and linear source bimodules }
\author{Michael Livesey\footnote{School of Mathematics, University of Manchester, Manchester, M13 9PL, United Kingdom. Email: michael.livesey@manchester.ac.uk, orcid: 0000-0003-3431-9020} and Claudio Marchi\footnote{School of Mathematics, University of Manchester, Manchester, M13 9PL, United Kingdom. Email: claudio.marchi@manchester.ac.uk, orcid: 0000-0003-2083-5016}}
\date{}

\begin{document}

\maketitle

\begin{abstract}

It is an open problem as to whether any bimodule inducing a Morita auto-equivalence of a block must have endopermutation source. We prove that, for blocks $b$ with normal defect groups in odd characteristic, a stronger result holds, namely that all such bimodules have linear source. We also prove the analogous result in characteristic $2$, provided that the defect group is of a specific, slightly restrictive, form.
\end{abstract}

\section{Introduction}

Let $\cO$ be a complete discrete valuation ring, with $k:=\cO/J(\cO)$ an algebraically closed field of characteristic $p>0$ and $K$, a field of characteristic zero, the field of fractions of $\cO$. Let $H$ be a finite group. In this setting $K$ will always be large enough, meaning that it contains all $\vert H\vert^{\nth}$ roots of unity. For the remainder of the introduction let $b$ be a block of $H$, by which we mean a block of $\cO H$.
\newline
\newline
The Picard group $\Pic(b)$ of $b$ consists of isomorphism classes of invertible $b$-$b$-bimodules. If $M\in\Pic(b)$, then $M$ induces an $\cO$-linear Morita auto-equivalence of $b$ given by $M\otimes_b-$. There are three important subgroups of $\Pic(b)$ that will form the main area of interest for this article. (For more details on $\cT(b)$, $\cL(b)$ and $\cE(b)$ see \cite{bokeli20}.)
\begin{align*}
\cT(b)&=\left\lbrace [M]_{\sim}\in\Pic(b)\vert M\text{ has trivial source as an }\cO (H\times H)\text{-module}\right\rbrace\\
\cL(b)&=\left\lbrace [M]_{\sim}\in\Pic(b)\vert M\text{ has linear source as an }\cO (H\times H)\text{-module}\right\rbrace\\
\cE(b)&=\left\lbrace [M]_{\sim}\in\Pic(b)\vert M\text{ has endopermutation source as an }\cO (H\times H)\text{-module}\right\rbrace.
\end{align*}
Morita equivalences given by endopermutation source bimodules seem to be very common in practice, in fact there are no known examples of Morita equivalences of blocks given by a bimodule that does not have endopermutation source. It is, therefore, a very natural question to ask if all elements of $\Pic(b)$ have endopermutation source. If it were always the case that $\Pic(b)=\cE(b)$, then it would be known that $\Pic(b)$ is bounded in terms of a function of the order of the defect group (see \cite[Theorem 1.3]{lili20}). In fact, it is proved in \cite[Corollary 1.2]{ei19} that $\Pic(b)$ is at least always finite.
\newline
\newline
For blocks with normal defect group, it is already known that $\cE(b)=\cL(b)$ (see \cite[Theorem 1.5]{lili20}). Therefore, in this case we ask whether $\Pic(b)=\cL(b)$. In \cite[Theorem 6.3]{livesey19} this was shown to hold for blocks with normal abelian defect and abelian inertial quotient. We improve this result, removing any hypothesis on the structure of the defect group and the inertial quotient, in our main theorem:
\begin{thm*}[\ref{thm:main}]
Let $p>2$ and $b$ a block with normal defect group. Then $\Pic(b)=\cL(b)$.
\end{thm*}
The corresponding theorem for $p=2$ is currently out of reach of the authors using the methods outlined in this article. However, if the abelianisation of the defect group is sufficiently ``tall'' the theorem still holds:
\begin{thm*}[\ref{thm:main2}]
Let $p=2$ and $b$ a block with normal defect group $D$ such that $D/[D,D]$ has no direct factor isomorphic to $C_2$. Then $\Pic(b)=\cL(b)$.
\end{thm*}
The following notation will hold throughout the article. If $N\lhd H$ and $\chi\in\Irr(N)$, then we denote by $\Irr(H|\chi)$ the set of irreducible characters of $H$ appearing as constituents of $\chi\uparrow_N^H$. Similarly, we set $\Irr(b|\chi):=\Irr(b)\cap\Irr(H|\chi)$. Let $F\leq H$. For any $h\in H$ and $\chi\in \Irr(F)$, we define ${}^hF = hFh^{-1}$ and ${}^h\chi\in\Irr({}^hF)$ by ${}^h\chi(g)=\chi(h^{-1}gh)$, for all $g\in {}^hF$. If now $E\leq H$ normalises $F$, then $E_\chi$ will symbolise the stabiliser of $\chi$ in $E$ and if $\eta = {}^h\chi$, for some $h\in E$ and $\eta\in \Irr(F)$, then we write $\chi\sim_E\eta$. We also adopt all the analogous above notation for Brauer characters, where we replace $\Irr$ with $\IBr$.
\newline
\newline
$\langle,\rangle_H$ will denote the usual inner product on $\ZZ\Irr(H)$. If $\chi\in\Irr(H)$ is a linear character then $\cO_\chi$ will be the $\cO H$-module $\cO$ with action of $H$ defined through $\chi$. We use $1_H\in\Irr(H)$ to signify the trivial character of $H$, $e_b\in\cO H$ the block idempotent of $b$ and $e_\chi\in KH$ the character idempotent associated to any $\chi\in\Irr(H)$. We define $\overline{\phantom{A}}:\cO\to k$ to be the natural quotient map, $\overline{\phantom{A}}:\ZZ\Irr(H)\to \ZZ\IBr(H)$ the corresponding reduction modulo $p$ map and $\overline{M}$ to be $k\otimes_{\cO} M$, for any $\cO H$-module $M$. We adopt the convention that all $\cO H$-modules are finitely generated and free as $\cO$-modules.
\newline
\newline
The article is organised as follows. $\S$\ref{sec:ext} is concerned with Ext groups and how they can be used to distinguish certain subsets of characters of a particular block. In $\S$\ref{sec:main} we prove our main theorems. The proofs heavily rely on an application of Weiss' criterion.

\section{Ext groups}\label{sec:ext}

This section is concerned with using Ext groups to distinguish certain subsets of irreducible characters of a block with normal defect group. This will allow us to apply Weiss' criterion in $\S$\ref{sec:main} via Proposition~\ref{prop:indMor}. We first gather together various well known results on Ext groups with respect to group rings. In what follows all cohomology groups will be calculated over $\cO$, not $\ZZ$. Also all isomorphisms of Ext and cohomology groups will assumed to be isomorphisms as $\cO$ or $k$-modules and not just as abelian groups.

\begin{lem}\label{lem:ext_groups}
Let $H$ be a finite group.
\begin{enumerate}[(i)]
\item For any $i\in\NN_0$ and $\cO H$-modules $M_1,M_2$,
\begin{align*}
\Ext_{\cO H}^i(M_1,M_2)\simeq H^i(H,M_1^*\otimes_{\cO}M_2).
\end{align*}
\item For any $N\leq H$, $\cO N$-module $M_1$ and $\cO H$-module $M_2$,
\begin{align*}
\Ext_{\cO H}^i(M_1\uparrow_N^H,M_2)&\simeq \Ext_{\cO N}^i(M_1,M_2\downarrow_N^H),\\
\Ext_{\cO H}^i(M_2,M_1\uparrow_N^H)&\simeq \Ext_{\cO N}^i(M_2\downarrow_N^H,M_1).
\end{align*}
\item If $H=H_1\times H_2$, then for any $n\in\NN$, $\cO H_1$-modules $U_1,V_1$ and $\cO H_2$-module $U_2,V_2$, there exists a split, short exact sequence
\begin{align*}
0& \longrightarrow \bigoplus_{i+j=n}\Ext_{\cO H_1}^i(U_1,V_1)\otimes_{\cO} \Ext_{\cO H_2}^j(U_2,V_2)\longrightarrow \Ext_{\cO H}^n(U_1\otimes_{\cO}U_2 ,U_1\otimes_{\cO}U_2)\\
&\longrightarrow \bigoplus_{i+j=n+1}\Tor^{\cO}_1(\Ext_{\cO H_1}^i(U_1,V_1), \Ext_{\cO H_2}^j(U_2,V_2)) \longrightarrow 0.
\end{align*}
\item If $H$ is an abelian $p$-group, say $H\simeq C_{p^{n_1}}\times\dots\times C_{p^{n_t}}$, and $\lambda_1,\lambda_2\in\Irr(H)$, then
\begin{align*}
\Ext_{\cO H}^0(\cO_{\lambda_1},\cO_{\lambda_2})&=
\begin{cases}
\cO&\text{if }\lambda_1=\lambda_2\\
\{0\}&\text{otherwise}
\end{cases},\\
\Ext_{\cO H}^1(\cO_{\lambda_1},\cO_{\lambda_2})&=
\begin{cases}
\{0\}&\text{if }\lambda_1=\lambda_2\\
\cO/(1-\zeta)\cO&\text{otherwise}
\end{cases},\\
\Ext_{\cO H}^2(\cO_{\lambda_1},\cO_{\lambda_2})&=
\begin{cases}
\bigoplus\limits_{i=1}^t\cO/p^{n_i}\cO&\text{if }\lambda_1=\lambda_2\\
[\cO/(1-\zeta)\cO]^{\oplus(t-1)}&\text{otherwise}
\end{cases},
\end{align*}
where it is assumed, if $\lambda_1\neq\lambda_2$, that $\lambda_1.\lambda_2^{-1}$ has image $\{\zeta^i\}_{i\in\ZZ}$, for some $p^{\nth}$-power root of unity $\zeta$.
\item For any $\cO$-module $M$,
\begin{align*}
\Tor^{\cO}_1(\cO,M)=\Tor^{\cO}_1(M,\cO)=\{0\}.
\end{align*}
Furthermore,
\begin{align*}
\Tor^{\cO}_1(\cO/a\cO,\cO/b\cO)\simeq \cO/a\cO\otimes_{\cO}\cO/b\cO,
\end{align*}
for all non-zero $a,b\in\cO$.
\end{enumerate}
\end{lem}

\begin{proof}
\begin{enumerate}[(i)]
\item This is very well known.
\item This is Shapiro's Lemma.
\item Set $M_1=U_1^*\otimes_{\cO}V_1$ and $M_2=U_2^*\otimes_{\cO}V_2$. The K\"unneth formula for group cohomology gives a split, short exact sequence
\begin{align*}
0& \longrightarrow \bigoplus_{i+j=n} H^i(H_1,M_1)\otimes_{\cO} H^j(H_2,M_2)
\longrightarrow H^n(H,M_1\otimes_{\cO} M_2)\\
&\longrightarrow \bigoplus_{i+j=n+1}\Tor^{\cO}_1(H^i(H_1,M_1),H^j(H_2,M_2)) \longrightarrow 0.
\end{align*}
For an explicit reference see \cite[Exercise 6.1.8]{weibel95}. Also note that in~\cite[Theorem 3.5.6]{benson91} it states that the underlying ring need not be $\ZZ$ but merely a hereditary ring. The result now follows from part (i).
\item This follows easily from the well known description of cohomology groups for cyclic groups (e.g. see \cite[Corollary 3.5.2]{benson91}) and parts (i) and (iii).
\item The first claim follows simply from the fact that $\cO$ is respectively projective/flat over itself. For the second claim we recall that
\begin{align*}
\Tor^{\cO}_1\left(\cO/a\cO,M\right)\simeq\Ann_M(a):=\left\lbrace x\in M\vert ax=0\right\rbrace,
\end{align*}
for any non-zero (divisor) $a\in\cO$ and $\cO$-module $M$. However, if $b\in\cO$ is also non-zero then 
\begin{align*}
\Ann_{\cO/b\cO}(a)\simeq \cO/c\cO \simeq \cO/a\cO\otimes_{\cO}\cO/b\cO,
\end{align*}
where $c$ is either $a$ or $b$ depending on which has the smaller valuation with respect to $J(\cO)$.
\end{enumerate}
\end{proof}

We now introduce a specific set up that will hold for the remainder of this section. Set $G=D\rtimes E$, where $D$ is a finite $p$-group and $E$ a finite $p'$-group. $Z:=C_E(D)$ is a cyclic, central subgroup of $E$ with quotient $L:=E/Z$. We set $B:=\cO(D\rtimes E)e_\varphi$, for some faithful $\varphi\in\Irr(Z)$. Since $D\lhd G$, any block idempotent of $\cO G$ must be supported on $C_G(D)=Z(D)\times Z$ and hence $B$ is indeed a block.
\newline
\newline
We will require this general set up, where $D$ is allowed to be non-abelian, in $\S$\ref{sec:main}. However, for the remainder of this section all the above will hold with the added assumption that $D$ is abelian. In this case, we set $D_1=[D,E]$ and $D_2=C_D(E)$. By~\cite[Theorem 2.3]{gor80}, we have a decomposition $D=D_1\times D_2$. Say $D_1\simeq C_{p^{n_1}}\times\dots\times C_{p^{n_u}}$ and $D_2\simeq C_{p^{n_{u+1}}}\times\dots\times C_{p^{n_t}}$, for some $t\geq u\in\NN_0$ and $n_i\in\NN$. We also demand that all the $n_i>1$, when $p=2$. We record all these assumptions below.

\begin{assumption}\label{ass}
\begin{enumerate}[(i)]
\item $B=\cO(D\rtimes E)e_\varphi$ but $D$ is not assumed to abelian. This set up will be used in $\S$\ref{sec:main}.
\item In addition to (i), $D$ is assumed to be abelian and if $p=2$ we assume that $n_i>1$, for all $1\leq i\leq t$, i.e. $D$ has no direct factor isomorphic to $C_2$. This set up will hold for the remainder of this section.
\end{enumerate}
\end{assumption}

We can describe $\Irr(B)$ very precisely. For any $\lambda\in\Irr(D)$ and $\chi\in\Irr(E_\lambda|\varphi)$ we define $(\lambda,\chi)\in\Irr(D\rtimes E_\lambda|\varphi)$ by
\[
(\lambda,\chi)(xg)=\lambda(x)\chi(g)\text{ for all }x\in D,g\in E_{\lambda}.
\]
Then
\begin{align}\label{algn:chars_D_abelian}
\Irr(B)=\left\{(\lambda,\chi)\uparrow_{D\rtimes E_{\lambda}}^G|\lambda\in\Irr(D),\chi\in\Irr(E_\lambda|\varphi)\right\}.
\end{align}
This was proved in \cite[Lemma 3.2]{livesey19} for $L$ abelian but the proof is identical in this more general setting.

\begin{rem}\label{rem:brauer}
Since $D\lhd G$, $D$ is in the kernel of every Brauer character of $B$ and so we can identify $\IBr(B)$ with $\Irr(E|\varphi)$. Furthermore, through this identification, we can identify the reduction map $\ZZ\Irr(B)\to \ZZ\IBr(B)$ with the restriction map $\ZZ\Irr(B)\to \ZZ\Irr(E|\varphi)$.
\end{rem}

Next we need a lemma concerning $\cO G$-modules with linear source that realise irreducible characters of $B$.

\begin{lem}\label{lem:char_realise}
Let $\eta\in \Irr(B)$.
\begin{enumerate}[(i)]
\item There exists a unique, up to isomorphism, linear source $\cO G$-module $M_\eta$ such that $K\otimes_{\cO}M_\eta$ affords $\eta$.
\item If $\eta$ lifts a Brauer character then $M_\eta$ is the unique, up to isomorphism, $\cO G$-module such that $K\otimes_{\cO}M_\eta$ affords $\eta$.
\end{enumerate}
\end{lem}

\begin{proof}
\begin{enumerate}[(i)]
\item Suppose $\eta=(\lambda,\chi)\uparrow_{D\rtimes E_\lambda}^G$, for some $\lambda\in \Irr(D)$ and $\chi\in \Irr(E_\lambda|\varphi)$. Define $V_\chi$ to be an $\cO E_\lambda$-module such that $K\otimes_{\cO}V_\chi$ affords $\chi$. We can extend $V_\chi$ to $\cO(D\rtimes E_\lambda)$ by letting each $x\in D$ act via scalar multiplication by $\lambda(x)$. (Since $E_\lambda$ is a $p'$-group, $V_\chi$ is uniquely defined by the above.) Now set $M_\eta=V_\chi\uparrow_{D\rtimes E_\lambda}^G$. Certainly $V_\chi$ and therefore $M_\eta$ has linear source.
\newline
\newline
Next let $M_\eta'$ be another linear source $\cO G$-module such that $K\otimes_{\cO}M_\eta'$ affords $\eta$. Since any source of $M_\eta'$ must be of the form $\cO_{\lambda'}$ for some $\lambda'\sim_E \lambda$, $M_\eta'$ must be a direct summand of $\cO_\lambda\uparrow_D^G$. Now each summand of $\cO_\lambda\uparrow_D^{D\rtimes E_\lambda}$ is of the form $V_{\chi'}$ from the above paragraph, for some $\chi'\in\Irr(E_\lambda)$. Since each $V_{\chi'}\uparrow_{D\rtimes E_\lambda}^G$ is indecomposable, $M_\eta'\simeq V_{\chi'}\uparrow_{D\rtimes E_\lambda}^G$ for some $\chi'\in\Irr(E_\lambda)$. Fixing said $\chi'$, we must have
\begin{align*}
1&=\langle(\lambda,\chi')\uparrow_{D\rtimes E_\lambda}^G,(\lambda,\chi)\uparrow_{D\rtimes E_\lambda}^G\rangle_G=\langle(\lambda,\chi'),(\lambda,\chi)\uparrow_{D\rtimes E_\lambda}^G\downarrow^G_{D\rtimes E_\lambda}\rangle_{D\rtimes E_\lambda}\\
&=\langle(\lambda,\chi'),(\lambda,\chi)\rangle_{D\rtimes E_\lambda},
\end{align*}
where the final equality follows since $(\lambda,\chi)\downarrow^{D\rtimes E_\lambda}_D=|E_\lambda|.\lambda$ and the stabiliser of $\lambda$ in $G$ is $D\rtimes E_\lambda$. In particular,
\begin{align*}
1\leq\langle(\lambda,\chi')\downarrow^{D\rtimes E_\lambda}_{E_\lambda},(\lambda,\chi)\downarrow^{D\rtimes E_\lambda}_{E_\lambda}\rangle_{E_\lambda}=\langle\chi',\chi\rangle_{E_\lambda}.
\end{align*}
Therefore, $\chi'=\chi$ and $M_\eta'\simeq V_\chi\uparrow_{D\rtimes E_\lambda}^G=M_\eta$, as desired.
\item Suppose $\eta$ lifts $\mu\in\IBr(G)$ and $U$ is an $\cO G$-module such that $K\otimes_{\cO}U$ affords $\eta$. Let $P_\mu$ be the projective indecomposable $\cO G$-module corresponding to $\mu$. So $\overline{P}_\mu$ is the projective indecomposable $kG$-module corresponding to $\mu$. In particular, we have a surjective $\cO G$-module homomorphism $f:P_\mu\to \overline{U}$. Therefore, since $P_\mu$ is projective, there exists an $\cO G$-module homomorphism $h:P_\mu\to U$ such that $\overline{\phantom{A}}\circ h=f$ and, by Nakayama's lemma, $h$ must surject onto $U$. The proof is completed by observing that, since $\eta$ lifts $\mu$, $\eta$ appears with multiplicity one in the character of $K\otimes_{\cO}P_\mu$ and so $P_\mu$ has a unique $\cO$-free quotient whose $K$-span affords $\eta$.
\end{enumerate}
\end{proof}

We adopt the notation of $M_\eta$, from the above lemma, for the remainder of this section. In particular, for $n\in\NN_0$ and $\chi,\eta\in\Irr(G)$, $\Ext_{\cO G}^n(\chi,\eta)$ will denote $\Ext_{\cO G}^n(M_\chi,M_\eta)$. We use an analogous convention for $\Ext_{kG}^n(\mu,\psi)$, where $\mu,\psi\in \IBr(B)$.

\begin{lem}\label{lem:kuenn} Let $\chi=\alpha\otimes \theta_{\chi}$, $\eta=\beta\otimes \theta_{\eta}\in\Irr(B)$, for some $\alpha,\beta\in\Irr(D_1\rtimes E|\varphi)$ and $\theta_\chi,\theta_\eta\in\Irr(D_2)$.
\begin{enumerate}[(i)]
\item For $0\leq i\leq 2$, each $\Ext_{\cO G}^i(\chi,\eta)$ is a finite direct sum of $\cO$-modules of the form $\cO/a\cO$, for some $a\in \cO$. Moreover, if $i=0$, all the $a$'s can be chosen to be zero, if $i=1$, then all the $a$'s can be chosen to be of the form $1-\zeta$, for some non-trivial $p^{\nth}$-power roots of unity $\zeta$ and if $i=2$, then all the $a$'s can be chosen to be non-zero. In particular, $\Ext_{\cO G}^i(\chi,\eta)$ has a direct summand isomorphic to $\cO$ only if $i=0$ and in this case
\begin{align*}
\Ext_{\cO G}^0(\chi,\eta)\simeq \cO^{\oplus\delta_{\chi,\eta}}.
\end{align*}
\item If $\theta_{\chi}=\theta_{\eta}$, then
\begin{align*}
&\Ext^2_{\cO G}(\chi,\eta)\simeq \left[\bigoplus_{i=u+1}^t \faktor{\cO}{p^{n_i}\cO}\otimes_{\cO} \Ext_{\cO(D_1\rtimes E)}^0(\alpha,\beta)\right]\\
&\oplus \Ext_{\cO(D_1\rtimes E)}^2(\alpha,\beta)\oplus\left[\bigoplus_{i=u+1}^t \faktor{\cO}{p^{n_i}\cO}\otimes_{\cO}\Ext_{\cO(D_1\rtimes E)}^1(\alpha,\beta)\right].
\end{align*}
\item If $\theta_{\chi}\neq \theta_{\eta}$ and $\theta_{\chi}^{-1}.\theta_{\eta}$ has image $\{\zeta^i\}_{i\in\ZZ}$, for some $p^{\nth}$-power root of unity $\zeta$, then
\begin{align*}
&\Ext^2_{\cO G}(\chi,\eta)\simeq \left[\faktor{\cO}{(1-\zeta)\cO}\otimes_{\cO}\Ext_{\cO(D_1\rtimes E)}^0(\alpha,\beta)\right]^{\oplus(t-u-1)}\\
\oplus&\left[\faktor{\cO}{(1-\zeta)\cO}\otimes_{\cO}\Ext_{\cO(D_1\rtimes E)}^1(\alpha,\beta)\right]^{\oplus (t-u)}
\oplus\left[\faktor{\cO}{(1-\zeta)\cO}\otimes_{\cO}\Ext_{\cO(D_1\rtimes E)}^2(\alpha,\beta)\right].
\end{align*}
\end{enumerate}
\end{lem}

\begin{proof}
\begin{enumerate}[(i)]
\item 

Note that for any $0\leq i\leq 2$ and $\lambda_1,\lambda_2\in\Irr(D)$, by Lemma \ref{lem:ext_groups}(iii),
\begin{align*}
\Ext_{\cO G}^i(\cO_{\lambda_1}\uparrow_D^G,\cO_{\lambda_2}\uparrow_D^G)\simeq \Ext_{\cO D}^i(\cO_{\lambda_1},\cO_{\lambda_2}\uparrow_D^G\downarrow^G_D).
\end{align*}
The claim now follows, by Lemma \ref{lem:ext_groups}(iv), since $M_\chi$ and $M_\eta$ both have linear source. That $\Ext_{\cO G}^0$ is of the desired form, can be seen by noting that $\Ext_{\cO G}^0=\Hom_{\cO G}$.
\item[(ii),(iii)] Applying Lemma~\ref{lem:ext_groups}(iii) we get the split, short exact sequence
\begin{align}
\begin{split}\label{tikz:kunn}
0 &\longrightarrow \bigoplus_{i+j=2} \Ext_{\cO(D\rtimes E)}^i(\alpha,\beta)\otimes_{\cO} \Ext_{\cO D_2}^j(\theta_{\chi},\theta_{\eta})
\longrightarrow \Ext_{\cO G}^2(\alpha\otimes\theta_\chi,\beta\otimes\theta_{\eta})\\
&\longrightarrow \bigoplus_{i+j=3}\Tor^{\cO}_1(\Ext_{\cO(D\rtimes E)}^i(\alpha,\beta),\Ext_{\cO D_2}^j(\theta_{\chi},\theta_{\eta})) \longrightarrow 0.
\end{split}
\end{align}
We now need only apply Lemma \ref{lem:ext_groups}(iv) and (v) and part (i) of the current lemma to (\ref{tikz:kunn}) to obtain the result.
\end{enumerate}
\end{proof}

\begin{lem}\label{lem:Ext1_Ext2}
\begin{enumerate}[(i)]
\item Let $\chi,\eta\in\Irr(B)$ be such that $\overline{\chi}$ and $\overline{\eta}$ have no irreducible constituents in common. Then $$k\otimes_{\cO}\Ext_{\cO G}^2(\chi,\eta)\simeq \Ext_{kG}^1(\overline{\chi},\overline{\eta}).$$
\item The graph with vertices labelled by $\IBr(B)$ and an edge between $\mu,\psi\in\IBr(B)$ if $\Ext_{kG}^1(\mu,\psi)\neq \{0\}$, is connected.
\end{enumerate}
\end{lem}

\begin{proof}
\begin{enumerate}[(i)]
\item Consider a projective resolution for $M_\chi$
\[
\cdots\rightarrow P_1\rightarrow P_0\rightarrow M_{\chi}\rightarrow 0,
\] 
and let $X$ be the corresponding Hom chain complex
\[
\cdots\leftarrow\Hom_{\cO G}(P_1,M_{\eta})\leftarrow\Hom_{\cO G}(P_0, M_{\eta})\leftarrow 0,
\]
where we consider $\Hom_{\cO G}(P_i,M_{\eta})$ to be in position $-i$. By \cite[Theorem 2.7.4]{carlson03}, we have an exact sequence of $k$-modules,
\begin{align}\label{algn:ses}
0\rightarrow H_{-1}(X)\otimes_{\cO}k\rightarrow H_{-1}(X\otimes_{\cO} k)\rightarrow\Tor_1^{\cO}(H_{-2}(X),k)\rightarrow 0.
\end{align}
Of course, $H_{-1}(X)\simeq\Ext^1_{\cO G}(\chi,\eta)$ and $H_{-2}(X)\simeq \Ext^2(\chi,\eta)$. By Lemma \ref{lem:ext_groups}(v) and Lemma \ref{lem:kuenn}(i),
\begin{align*}
\Tor_1^{\cO}(\Ext^2(\chi,\eta),k)\simeq \Ext^2(\chi,\eta)\otimes_{\cO}k.
\end{align*}
It remains to show that the middle term is $\Ext^1_{kG}(\overline{\chi},\overline{\eta})$ and $\Ext^1_{\cO G}(\chi,\eta)=\{0\}$. For the first claim we just need to prove that $X\otimes_{\cO}k$ is the chain complex
\[
\cdots\leftarrow\Hom_{\cO G}(P_1\otimes_{\cO}k,M_{\eta}\otimes_{\cO}k)\leftarrow\Hom_{\cO G}(P_0\otimes_{\cO}k, M_{\eta}\otimes_{\cO}k)\leftarrow 0.
\]
In other words we need to show that every element of
\begin{align*}
\Hom_{kG}(P_i\otimes_{\cO} k,M_\eta\otimes_{\cO}k)
\end{align*}
lifts to an element of $\Hom_{\cO G}(P_i, M_{\eta})$. If $G=D$, then the claim follows, since each $P_i$ is a direct sum of $\cO D$'s,
\begin{align*}
\Hom_{\cO G}(\cO D,\cO_\lambda)=\langle g\mapsto \lambda(g) \rangle_{\cO},
\end{align*}
for any $\lambda\in\Irr(D)$, and
\begin{align*}
\Hom_{kG}(kD,k)=\langle g\mapsto 1 \rangle_k.
\end{align*}
The general case follows from tracing through the series of isomorphisms
\begin{align*}
&\Hom_{\cO G}\left(\cO D\uparrow_{D}^G,\cO_\lambda\uparrow_D^G\right) \otimes_{\cO}k\simeq \Hom_{\cO G}\left(\cO D,\cO_\lambda\uparrow_D^G\downarrow_{D}^G\right)\otimes_{\cO}k\\
\simeq&\Hom_{kG}\left(kD,k_\lambda\uparrow_D^G\downarrow_{D}^G\right) \simeq \Hom_{kG}\left(\cO D\uparrow_{D}^G\otimes_{\cO}k,\cO_\lambda\uparrow_D^G\otimes_{\cO}k\right)
\end{align*}
and noting that every projective $\cO G$-module is a direct summand of $\cO D\uparrow_{D}^G$ and every $M_\eta$ is a direct summand of some $\cO_\lambda\uparrow_D^G$. (Note the first and third isomorphisms follow from Frobenius reciprocity and the second since we already know our claim holds for $G=D$.)
\newline
\newline
We now show that $\Ext^1_{\cO G}(\chi,\eta)=\{0\}$ and this will conclude the proof. Take $P_{\chi}$ a projective cover of $\chi$, then we have an exact sequence
\begin{align*}
0 &\longrightarrow \Hom_{\cO G}(M_\chi,M_\eta) \longrightarrow  \Hom_{\cO G}(P_{\chi},M_\eta) \longrightarrow   \Hom_{\cO G}(\Omega^1(\chi),M_\eta)\\
&\longrightarrow \Ext^1_{\cO G}(M_\chi,M_\eta) \longrightarrow  \Ext^1_{\cO G}(P_{\chi},M_\eta) \longrightarrow \dots
\end{align*}
where $\Omega^1(\chi)$ is the kernel of the natural map $P_\chi\to M_\chi$. We know that $\Hom_{\cO G}(P_{\chi},M_\eta)$ and $\Hom_{\cO G}(\Omega^1(\chi),M_\eta)$ are both zero, since the hypotheses of the lemma ensure that $\eta$ has multiplicity zero in the characters of $K\otimes_{\cO}P_{\chi}$ and $K\otimes_{\cO}\Omega^1(\chi)$. Moreover, as $P_\chi$ is projective, $\Ext_{\cO G}^1(P_\chi,M_\eta)$ is zero and thus $\Ext^1_{\cO G}(\chi,\eta)=\Ext^1_{\cO G}(M_\chi,M_\eta)=\{0\}$ as required.
\item This is \cite[Proposition 4.13.3]{alp93}.
\end{enumerate}
\end{proof}

Before proceeding we need to note the following. If $\zeta\in\cO$ is a primitive $(p^n)^{\nth}$-root of unity, for some $n\in\NN$, then
\begin{align*}
\prod_{\substack{i=1, \\ p\nmid i}}^{p^n-1}(X-\zeta^i)=(X^{p^n}-1)/(X^{p^{n-1}}-1)=\sum_{i=0}^{p-1}X^{ip^{n-1}}\in\ZZ[X].
\end{align*}
In particular,
\begin{align*}
\prod_{\substack{i=1, \\ p\nmid i}}^{p^n-1}(1-\zeta^i)=p
\end{align*}
and so $1-\zeta\in p\cO$ if and only if $p=2$ and $n=2$. This is ultimately the reason that we do not have as strong a theorem in the $p=2$ case.
\newline
\newline
We need one small lemma before continuing.

\begin{lem}\label{lem:fff}
If $\lambda\in \Irr(D_1)$ is $E$-stable, then $\lambda=1_{D_1}$.
\end{lem}

\begin{proof}
This was proved in \cite[Lemma 2.4]{livesey19} under the assumption that $L$ is abelian. However, this fact was totally unused in the proof.
\end{proof}

\begin{lem}\label{lem:Ext2_non-zero}
Let $\lambda_1,\lambda_2\in\Irr(D)$, $\chi_1\in\Irr(E_{\lambda_1}\vert\varphi)$ and $\chi_2\in\Irr(E_{\lambda_2}\vert\varphi)$.
\begin{enumerate}[(i)]
\item If $p>2$ and $$\Ext_{\cO G}^2((\lambda_1,\chi_1)\uparrow_{D\rtimes E_{\lambda_1}}^G,(\lambda_2,\chi_2)\uparrow_{D\rtimes E_{\lambda_2}}^G)\simeq \bigoplus_{i=1}^s\cO/p^{m_i}\cO,$$ for some $s\in\NN$ and $m_1,\dots,m_s\in\NN$, then $\lambda_1\sim_E \lambda_2$.
\item If $p=2$ and $$\Ext_{\cO G}^2((\lambda_1,\chi_1)\uparrow_{D\rtimes E_{\lambda_1}}^G,(\lambda_2,\chi_2)\uparrow_{D\rtimes E_{\lambda_2}}^G)\simeq \bigoplus_{i=1}^s\cO/p^{m_i}\cO,$$ for some $s\in\NN$ and $m_1,\dots,m_s\in\NN_{>1}$, then $\lambda_1\sim_E \lambda_2$.
\end{enumerate}
\end{lem}

\begin{proof}
We prove only the $p>2$ case. The $p=2$ case follows in a similar fashion. If $\zeta\in\cO$ is a primitive $(p^n)^{\nth}$-root of unity, for some $n\in\NN$, the only difference between the two cases is that, due to the comments preceding the lemma, $1-\zeta \notin p\cO$ if $p>2$ but when $p=2$ we can have $1-\zeta \in 2\cO$ but not $1-\zeta\in 4\cO$.
\newline
\newline
The claim follows immediately if $(\lambda_1,\chi_1)\uparrow_{D\rtimes E_{\lambda_1}}^G=(\lambda_2,\chi_2)\uparrow_{D\rtimes E_{\lambda_2}}^G$ so let's assume from now on that $(\lambda_1,\chi_1)\uparrow_{D\rtimes E_{\lambda_1}}^G\neq(\lambda_2,\chi_2)\uparrow_{D\rtimes E_{\lambda_2}}^G$.
\newline
\newline
Suppose $\lambda_i=\alpha_i\otimes \beta_i$, for $i=1,2$, $\alpha_i\in\Irr(D_1)$ and $\beta_i\in\Irr(D_2)$. Then
\begin{align*}
&\Ext_{\cO G}^2((\lambda_1,\chi_1)\uparrow_{D\rtimes E_{\lambda_1}}^G,(\lambda_2,\chi_2)\uparrow_{D\rtimes E_{\lambda_2}}^G)\\
\simeq& \Ext_{\cO G}^2((\alpha_1,\chi_1)\uparrow_{D_1\rtimes E_{\alpha_1}}^{D_1\rtimes E}\otimes\beta_1,(\alpha_2,\chi_2)\uparrow_{D_1\rtimes E_{\alpha_2}}^{D_1\rtimes E}\otimes\beta_2).
\end{align*}
Therefore, by the hypotheses of the lemma and the comments preceding it, Lemma~\ref{lem:kuenn}(iii) gives that $\beta_1=\beta_2$. Furthermore, by Lemma~\ref{lem:kuenn}(ii),
\begin{align*}
\Ext_{\cO(D_1\rtimes E)}^2((\alpha_1,\chi_1)\uparrow_{D_1\rtimes E_{\alpha_1}}^{D_1\rtimes E},(\alpha_2,\chi_2)\uparrow_{D_1\rtimes E_{\alpha_2}}^{D_1\rtimes E})\simeq \bigoplus_{i=1}^s\cO/p^{m_i}\cO.
\end{align*}
Note that, since we are assuming the two characters are different, the $\Ext^0$ term in Lemma~\ref{lem:kuenn}(ii) is just $\{0\}$. Also, by Lemma~\ref{lem:kuenn}(i), the hypotheses of the lemma and the comments preceding it, the $\Ext^1$ term in Lemma~\ref{lem:kuenn}(ii) is also $\{0\}$. We now need to show that $\alpha_1\sim_E\alpha_2$. Lemma \ref{lem:ext_groups}(ii) gives
\begin{align}
\begin{split}\label{algn:frob_recip}
&\Ext_{\cO(D_1\rtimes E)}^2((\alpha_1,\chi_1)\uparrow_{D_1\rtimes E_{\alpha_1}}^{D_1\rtimes E},(\alpha_2,\chi_2)\uparrow_{D_1\rtimes E_{\alpha_2}}^{D_1\rtimes E})\\
\simeq& \Ext_{\cO(D_1\rtimes E_{\alpha_1})}^2((\alpha_1,\chi_1),(\alpha_2,\chi_2)\uparrow_{D_1\rtimes E_{\alpha_2}}^{D_1\rtimes E}\downarrow_{D_1\rtimes E_{\alpha_1}}^{D_1\rtimes E})
\end{split}
\end{align}
and, by the Mackey formula,
\begin{align*}
(\alpha_2,\chi_2)\uparrow_{D\rtimes E_{\alpha_2}}^{D_1\rtimes E}\downarrow_{D\rtimes E_{\alpha_1}}^{D_1\rtimes E}=\sum_i({}^{g_i}\alpha_2,\chi_{g_i})\uparrow_{D_1\rtimes ({}^{g_i}E_{\alpha_2}\cap E_{\alpha_1})}^{D_1\rtimes E_{\alpha_1}},
\end{align*}
for some $g_i$'s in $E$ and $\chi_{g_i}\in\Irr({}^{g_i}E_{\alpha_2}\cap E_{\alpha_1},\varphi)$. Therefore, by the hypotheses of the lemma, there must exist some $i$ with
\begin{align}\label{algn:ind_step}
\Ext_{\cO (D_1\rtimes E)}^2((\alpha_1,\chi_1),({}^{g_i}\alpha_2,\chi_{g_i})\uparrow_{D_1\rtimes ({}^{g_i}E_{\alpha_2}\cap E_{\alpha_1})}^{D_1\rtimes E_{\alpha_1}})\simeq\bigoplus_{i=1}^{s'}\cO/p^{m_i'}\cO,
\end{align}
for some $s'\in\NN$ and $m_1',\dots,m_s'\in\NN$. 
We now proceed by induction on $|E|$. We first note that the result holds in the base case $E=Z$ by Lemma~\ref{lem:ext_groups}(iii) and (iv). Next, we see that we can perform the inductive step, using (\ref{algn:ind_step}), unless $E_{\alpha_1}=E$. In this case, by Lemma~\ref{lem:fff}, $\alpha_1=1_{D_1}$. We now reverse the roles of $(\alpha_1,\chi_1)\uparrow_{D_1\rtimes E_{\alpha_1}}^{D_1\rtimes E}$ and $(\alpha_2,\chi_2)\uparrow_{D_1\rtimes E_{\alpha_2}}^{D_1\rtimes E}$ in (\ref{algn:frob_recip}), i.e.
\begin{align*}
&\Ext_{\cO(D_1\rtimes E)}^2((\alpha_1,\chi_1)\uparrow_{D_1\rtimes E_{\alpha_1}}^{D_1\rtimes E},(\alpha_2,\chi_2)\uparrow_{D_1\rtimes E_{\alpha_2}}^{D_1\rtimes E})\\
\simeq& \Ext_{\cO(D_1\rtimes E_{\alpha_2})}^2((\alpha_1,\chi_1)\uparrow_{D_1\rtimes E_{\alpha_1}}^{D_1\rtimes E}\downarrow_{D_1\rtimes E_{\alpha_2}}^{D_1\rtimes E},(\alpha_2,\chi_2)),
\end{align*}
and proceed as above. This time the claim follows by induction unless $E_{\alpha_2}=E$. In this case we have $\alpha_1=\alpha_2=1_{D_1}$ and the result is proved.
\end{proof}

\begin{property}\label{property:good}
We say a subset $X\subseteq \Irr(B)$ is good if it satisfies the following properties:
\begin{enumerate}[(i)]
\item Each $\chi\in X$ is a lift of a Brauer character and each $\psi\in\IBr(B)$ has exactly one lift in $X$.
\item If $p>2$, then for every $\chi_1,\chi_2\in X$, there exist $s\in \NN_0$ and $m_1,\dots,m_s\in\NN$ such that $$\Ext_{\cO G}^2(\chi_1,\chi_2)\simeq \bigoplus_{i=1}^s\cO/p^{m_i}\cO.$$
If $p=2$, then for every $\chi_1,\chi_2\in X$, there exist $s\in \NN_0$ and $m_1,\dots,m_s\in\NN_{>1}$ such that $$\Ext_{\cO G}^2(\chi_1,\chi_2)\simeq \bigoplus_{i=1}^s\cO/p^{m_i}\cO.$$
\end{enumerate}
\end{property}

\begin{prop}\label{prop:good}
A subset $X\subseteq \Irr(B)$ is good if and only if there exists some $\theta\in \Irr(D_2)$ such that
\begin{align*}
X=\{(1_{D_1},\chi)\otimes \theta|\chi\in \Irr(E\vert\varphi)\}=\Irr(B|1_{D_1}\otimes \theta).
\end{align*}
\end{prop}

\begin{proof}
As with Lemma~\ref{lem:Ext2_non-zero}, we only prove the $p>2$ case. The proposition is proved for $p=2$ in an identical way, once we have taken into account Assumption~\ref{ass} and the differences in Lemma~\ref{lem:Ext2_non-zero} and Property~\ref{property:good}, in the $p=2$ case.
\newline
\newline
We first assume $X$ is of the desired form and let $\theta$ be the relevant irreducible character of $D_2$. By Remark \ref{rem:brauer}, $X$ satisfies Property~\ref{property:good}(i). Now, by Lemma \ref{lem:ext_groups}(ii) and (iv),
\begin{align*}
&\Ext_{\cO G}^2((1_{D_1}\otimes\theta)\uparrow_D^G,(1_{D_1}\otimes\theta)\uparrow_D^G) \simeq \Ext_{\cO D}^2((1_{D_1}\otimes\theta),(1_{D_1}\otimes\theta)\uparrow_D^G\downarrow^G_D)\\
\simeq& \Ext_{\cO D}^2(1_{D_1}\otimes\theta,1_{D_1}\otimes\theta)^{\oplus [G:D]}\simeq \bigoplus_{i=1}^t(\cO/p^{n_i}\cO)^{\oplus [G:D]}.
\end{align*}
Since every $\chi\in X$ is an irreducible constituent of $(1_{D_1}\otimes\theta)\uparrow_D^G$, $X$ satisfies Property~\ref{property:good}(ii) and $X$ is good.
\newline
\newline
For the remainder of the proof we identify $\IBr(B)$ with $\Irr(E|\varphi)$ (and do other analogous identifications), using Remark~\ref{rem:brauer}.
\newline
\newline
For the converse, suppose $X$ is good. For any $\mu\in\IBr(B)=\Irr(E|\varphi)$, let $(\lambda_\mu,\chi_\mu)\uparrow_{G_{\lambda_\mu}}^G\in X$ be the unique lift in $X$ of $\mu$, for some $\lambda_\mu \in\Irr(D)$ and $\chi_\mu\in\Irr(E_{\lambda_\mu}\vert\varphi)$. In particular,
\begin{align}\label{algn:ind_bij}
\mu = (\lambda_\mu,\chi_\mu)\uparrow_{D\rtimes E_{\lambda_\mu}}^G\downarrow^G_E=\chi_\mu\uparrow_{E_{\lambda_\mu}}^E,
\end{align}
for each $\mu\in\Irr(E|\varphi)$. Next, let $\mu\neq\psi\in\Irr(E|\varphi)$ satisfy $\Ext_{kG}^1(\mu,\psi)\neq \{0\}$. Since $X$ satisfies Property~\ref{property:good}(ii), 
\begin{align*}
\Ext_{\cO G}^2((\lambda_\mu,\chi_\mu)\uparrow_{G_{\lambda_\mu}}^G,(\lambda_\psi,\chi_\psi)\uparrow_{G_{\lambda_\psi}}^G)\simeq \bigoplus_{i=1}^s\cO/p^{m_i}\cO,
\end{align*}
for some $s\in\NN$ and $m_1,\dots,m_s\in\NN$. (Lemma~\ref{lem:Ext1_Ext2}(i) ensures that $s>0$.) Therefore, by Lemma~\ref{lem:Ext2_non-zero}, $\lambda_\mu \sim_E \lambda_\psi$. Now, by Lemma~\ref{lem:Ext1_Ext2}(ii) and the fact that $X$ satisfies Property~\ref{property:good}(i), we can say the same thing for any pair of characters in $X$. In other words, we may assume that all the $\lambda_\mu$'s are equal to some fixed $\lambda\in\Irr(D)$. Now, for not necessarily distinct $\mu,\psi\in\Irr(E|\varphi)$,
\begin{align*}
&\Ext_{\cO G}^2((\lambda,\chi_\mu) \uparrow_{G_\lambda}^G,(\lambda,\chi_\psi) \uparrow_{G_\lambda}^G) \simeq \Ext_{\cO G_\lambda}^2((\lambda,\chi_\mu),(\lambda,\chi_\psi) \uparrow_{G_\lambda}^G \downarrow_{G_\lambda}^G)\\
\simeq & \Ext_{\cO G_\lambda}^2\left((\lambda,\chi_\mu),\sum_{g\in E_\lambda \backslash E/E_\lambda}{}^g(\lambda,\chi_\psi) \downarrow^{{}^g G_\lambda}_{{}^g G_\lambda\cap G_\lambda}\uparrow_{{}^g G_\lambda\cap G_\lambda}^{G_\lambda}\right).
\end{align*}
Of course, the only $\lambda'\in\Irr(D)$ $E_\lambda$-conjugate to $\lambda$ is $\lambda$ itself. Therefore, the only irreducible constituent of $(\lambda,\chi_\psi) \uparrow_{G_\lambda}^G \downarrow_{G_\lambda}^G$ of the form $(\lambda',\chi')\uparrow_{G_{\lambda'}}^G$, for some $\lambda'\in\Irr(D)$ and $\chi'\in\Irr(E_{\lambda'}\vert\varphi)$, with $\lambda'$ $E_\lambda$-conjugate to $\lambda$ is $(\lambda,\chi_\psi)$, with multiplicity one. Therefore, since $X$ satisfies Property~\ref{property:good}(ii), Lemma~\ref{lem:Ext2_non-zero} applied to $\cO G_\lambda e_\varphi$ gives that
\begin{align*}
\Ext_{\cO G_\lambda}^2((\lambda,\chi_\mu) \uparrow_{G_\lambda}^G,(\lambda,\chi_\psi) \uparrow_{G_\lambda}^G) \simeq \Ext_{\cO G_\lambda}^2((\lambda,\chi_\mu),(\lambda,\chi_\psi))
\end{align*}
and
\begin{align}\label{algn:zero_Ext2}
\Ext_{\cO G_\lambda}^2((\lambda,\chi_\mu),{}^g(\lambda,\chi_\psi)\downarrow^{{}^g G_\lambda}_{{}^g G_\lambda\cap G_\lambda}\uparrow_{{}^g G_\lambda\cap G_\lambda}^{G_\lambda})=\{0\},
\end{align}
for any $g\in E\setminus E_\lambda$.
\newline
\newline
Now, since $X$ satisfies Property~\ref{property:good}(ii), (\ref{algn:ind_bij}) gives that
\begin{align*}
\delta_{\mu,\psi}&=\langle \mu,\psi\rangle_E=\langle\chi_\mu\uparrow_{E_\lambda}^E,\chi_\psi\uparrow_{E_\lambda}^E\rangle_E=\langle\chi_\mu,\chi_\psi\uparrow_{E_\lambda}^E\downarrow^E_{E_\lambda}\rangle_{E_\lambda}\\
&=\left\langle\chi_\mu,\sum_{g\in E_\lambda \backslash E/E_\lambda}{}^g\chi_\psi\downarrow^{{}^g E_\lambda}_{{}^g E_\lambda\cap E_\lambda}\uparrow_{{}^g E_\lambda\cap E_\lambda}^{E_\lambda}\right\rangle_{E_\lambda},
\end{align*}
for all $\mu,\psi\in\Irr(E|\varphi)$. In other words, for any $g\in E$,
\begin{align*}
\langle\chi_\mu,{}^g\chi_\psi\downarrow^{{}^g E_\lambda}_{{}^g E_\lambda\cap E_\lambda}\uparrow_{{}^g E_\lambda\cap E_\lambda}^{E_\lambda}\rangle_{E_\lambda}=
\begin{cases}
1 & \text{if }\mu=\psi\text{ and }g\in E_\lambda,\\
0 & \text{otherwise.}
\end{cases}
\end{align*}
Therefore, since
\begin{align*}
{}^g(\lambda,\chi_\psi)\downarrow^{{}^g G_\lambda}_{{}^g G_\lambda\cap G_\lambda}\uparrow_{{}^g G_\lambda\cap G_\lambda}^{G_\lambda}\downarrow^{G_\lambda}_{E_\lambda}={}^g\chi_\psi\downarrow^{{}^g E_\lambda}_{{}^g E_\lambda\cap E_\lambda}\uparrow_{{}^g E_\lambda\cap E_\lambda}^{E_\lambda},
\end{align*}
we can apply Lemma~\ref{lem:Ext1_Ext2}(i) to (\ref{algn:zero_Ext2}) to give
\begin{align}\label{algn:zero_Ext1}
\Ext_{kG_\lambda}^1(\chi_\mu,{}^g\chi_\psi\downarrow^{{}^g E_\lambda}_{{}^g E_\lambda\cap E_\lambda}\uparrow_{{}^g E_\lambda\cap E_\lambda}^{E_\lambda})=\{0\},
\end{align}
for any $\mu,\psi\in\Irr(E|\varphi)$ and $g\in E\setminus E_\lambda$. Certainly every $\eta\in\Irr(E_\lambda|\varphi)$ is an irreducible constituent of
\begin{align*}
\psi\downarrow^E_{E_\lambda}=\chi_\psi\uparrow_{E_\lambda}^E\downarrow^E_{E_\lambda}=\sum_{g\in E_\lambda\backslash E/E_\lambda}{}^g \chi_\psi\downarrow^{{}^gE_\lambda}_{{}^g E_\lambda\cap E_\lambda}\uparrow_{{}^g E_\lambda\cap E_\lambda}^{E_\lambda},
\end{align*}
for some $\psi\in\Irr(E|\varphi)$, where the first equality follows from (\ref{algn:ind_bij}). Therefore, (\ref{algn:zero_Ext1}) implies that $\Ext_{kG_\lambda}^1(\chi_\mu,\eta)\neq\{0\}$, for some $\eta\in\Irr(E_\lambda|\varphi)$, only if $\eta=\chi_\psi$, for some $\psi\in\Irr(E|\varphi)$. Lemma~\ref{lem:Ext1_Ext2}(ii) applied to the block $\cO (D\rtimes E_\lambda)e_\varphi$ now gives that
\begin{align*}
\Irr(E_\lambda\vert\varphi) = \{\chi_\mu|\mu\in\Irr(E|\varphi)\}.
\end{align*}
So, again by (\ref{algn:ind_bij}), we have a one-to-one bijection between $\Irr(E_\lambda\vert\varphi)$ and $\Irr(E\vert\varphi)$ given by induction. This implies $E_\lambda=E$, since
\begin{align*}
&[E:E_\lambda].\dim_K(KE_\lambda e_\varphi)=\dim_K(KEe_\varphi)=\sum_{\chi\in\Irr(E|\varphi)}\chi(1)^2\\
=&\sum_{\eta\in\Irr(E_\lambda|\varphi)}\eta\uparrow_{E_\lambda}^E(1)^2=[ E:E_\lambda]^2.\sum_{\eta\in\Irr(E_\lambda|\varphi)}\eta(1)^2=[E:E_\lambda]^2.\dim_K(KE_\lambda e_\varphi).
\end{align*}
Therefore, by Lemma~\ref{lem:fff}, $\lambda = 1_{D_1}\otimes \theta$, for some $\theta\in\Irr(D_2)$, as required.
\end{proof}

\section{Weiss' criterion and the main theorems}\label{sec:main}

Before proceeding we need to set up some notation. Let $b$ be a block of $\cO H$, for some finite group $H$ and $Q$ a normal $p$-subgroup of $H$. We denote by $b^Q$ the direct sum of blocks of $\cO(H/Q)$ dominated by $b$, that is those blocks not annihilated by the image of $e_b$ under the natural $\cO$-algebra homomorphism $\cO H\to \cO(H/Q)$. It will also be necessary to define $\cT(b)$ for $b$ a sum of blocks of $H$. Note our definition of $\Pic(b)$ already makes sense for $b$ a sum of blocks.
\begin{align*}
\cT(b)=&\left\lbrace [M]_{\sim}\in\Pic(b)\vert M\text{ is a direct sum of modules with trivial source}\right.\\
&\left.\text{as }\cO (H\times H)\text{-modules}\right\rbrace.
\end{align*}
The following Proposition was shown to be a consequence of Weiss' criterion in~\cite[Propositions 4.3,4.4]{eali19}. Weiss' criterion is a statement about permutation modules originally stated in~\cite[Theorem 2]{we88} but proved in its most general form in~\cite[Theorem 1.2]{mcsyza18}. It has, in recently years, been a very important tool for calculating Picard groups.

\begin{prop}\label{prop:indMor}$ $
\begin{enumerate}[(i)]
\item The inflation map $\Inf_{H/Q}^H:\Irr(H/Q)\to\Irr(H)$ induces a bijection between $\Irr(b^Q)$ and $\Irr(b\vert 1_Q)$.
\item Suppose $M$ is a $b$-$b$-bimodule inducing a Morita auto-equivalence of $b$ that permutes the elements of $\Irr(b\vert 1_Q)$.
\begin{enumerate}
\item Then ${}^QM$, the set of fixed points of $M$ under the left action of $Q$, induces a Morita auto-equivalence of $b^Q$. Furthermore, the permutation of $\Irr(b^Q)$ induced by ${}^QM$ is identical to the permutation that $M$ induces on $\Irr(b\vert 1_Q)$, once these two sets have been identified using part (i).
\item If ${}^QM\in\cT(b^Q)$, then $M\in\cT(b)$.
\end{enumerate}
\end{enumerate}
\end{prop}


We immediately drop all the assumptions on $b$ made just before Proposition~\ref{prop:indMor}. We are now ready to prove our main theorem.

\begin{thm}\label{thm:main}
Let $p>2$ and $b$ a block with normal defect group. Then $\Pic(b)=\cL(b)$.
\end{thm}

\begin{proof}
We first reduce to the situation that $b=B$ as in Assumption \ref{ass}(i), where the defect group of $b$ is isomorphic to $D$ and its inertial quotient is isomorphic to $L$. Indeed it follows from \cite[Theorem A]{ku85} that such a $B$ can be chosen to be Morita equivalent to $b$ and even source algebra equivalent by \cite[Theorem 6.14.1]{linckelmann18}. Note that, by \cite[Theorem 6.7.6(v)]{linckelmann18}, $L$ must be prime to $p$ and $Z$ must be as well, since otherwise $B$ has defect group larger than that of $b$. In other words $E$ is indeed a $p'$-group as assumed in Assumption \ref{ass}(i).
\newline
\newline
By~\cite[Lemma 2.8]{bokeli20}, the above equivalence yields isomorphisms $\Pic(b)\simeq \Pic(B)$ and $\cL(b)\simeq \cL(B)$. Therefore, since we are concerned only with $\Pic(b)$ and $\cL(b)$, we may assume that $b=B$.
\newline
\newline
Let $M\in\Pic(B)$. We first show that the elements of $\Irr(B|1_{[D,D]})$ get permuted by the permutation of $\Irr(B)$ induced by $M$. This claim will follow from~\cite[Lemme 1.6]{br90} once we have shown that $\Irr(B|1_{[D,D]})$ is precisely the subset of irreducible characters of $B$ of height zero, i.e. $\Irr(B|1_{[D,D]})$ is precisely the subset of irreducible characters of $B$ of $p'$-degree.
\newline
\newline
If $\eta\in\Irr(B|1_{[D,D]})$ then we can view $\eta\in \Irr(B^{[D,D]})$ via Proposition \ref{prop:indMor}(i) and hence as the character of a block with normal abelian defect group. After possibly applying the reduction from the beginning of the proof and~\cite[Lemme 1.6]{br90} again, it follows from the description around (\ref{algn:chars_D_abelian}) that $\eta$ has $p'$-degree.
\newline
\newline
Conversely if $\eta\in\Irr(B|\lambda)$, for some non-linear $\lambda\in\Irr(D)$, then $\eta\downarrow^G_D$ is a sum of conjugates of $\lambda$. In particular, $\lambda(1)\mid \eta(1)$ and $\eta$ does not have $p'$-degree.
\newline
\newline
We can now apply Proposition~\ref{prop:indMor}(ii)(a) with respect to $M$ and $[D,D]$ to obtain ${}^{[D,D]}M\in\Pic(B^{[D,D]})$. In this case~\cite[Corollary 4]{ku87} tells us that $B^{[D,D]}$ is just a single block of $\cO(G/[D,D])$. Actually, $B^{[D,D]}$ is of the form described in Assumption \ref{ass}(ii). Indeed, if $h\in E$ acts trivially on $D/[D,D]$, then $h$ must also act trivially on each factor group $\gamma_i(D)/\gamma_{i+1}(D)$ in the lower central series of $D$, so, by \cite[Corollary 5.3.3]{gor80}, the subgroup generated by $h$ in $\Aut(D)$ is a $p$-group, and, since $h$ has $p'$-order, $h\in Z$. In conclusion we have $C_{E}(D/[D,D])=Z$ and we adopt the slight abuse of notation $D/[D,D]=D_1\times D_2$ from $\S$\ref{sec:ext}.
\newline
\newline
Note that by Lemma~\ref{lem:char_realise}(ii), the notion of being good from Property~\ref{property:good} is a Morita invariant. Therefore, since $B^{[D,D]}$ is now of the desired form, we can apply Proposition~\ref{prop:good} to obtain that there exists some $\theta\in \Irr(D_2)$ such that $\Irr(B^{[D,D]}|1_{D/[D,D]})$ gets sent to $\Irr(B^{[D,D]}|1_{D_1}\otimes \theta)$ under the permutation of $\Irr(B^{[D,D]})$ induced by ${}^{[D,D]}M$. The last sentence in Proposition~\ref{prop:indMor}(ii)(a) now implies that $\Irr(B|1_D)$ gets sent to $\Irr(B|\Inf_{D/[D,D]}^D(1_{D_1}\otimes \theta))$ under the permutation of $\Irr(B)$ induced by $M$.
\newline
\newline
Set $\mathbf{D}_1$ to be the preimage of $D_1$ in $D$. Since $D_2\simeq G/(\mathbf{D}_1\rtimes E)$, we can set $\omega\in\Irr(G)$ to be the inflation of $\theta$ to $G$ and $M_{\omega^{-1}}\in \Pic(B)$ to be the $B$-$B$-bimodule inducing the Morita auto-equivalence given by tensoring with $\omega^{-1}$. Now $\Irr(B|1_D)$ gets permuted under the permutation of $\Irr(B)$ induced by $M_{\omega^{-1}}\otimes_B M$. We can therefore apply Proposition~\ref{prop:indMor}(ii) with respect to $M_{\omega^{-1}}\otimes_B M$ and $D$ to obtain that $M_{\omega^{-1}}\otimes_B M\in\cT(B)$. ($G/D$ is a $p'$-group and so certainly ${}^D(M_{\omega^{-1}}\otimes_B M)\in\cT(B^D)$.) Since $\cT(B)$ is a subgroup of $\cL(B)$, it remains only to show that $M_{\omega^{-1}}\in\cL(B)$. However,
\begin{align*}
M_{\omega^{-1}}=\cO_{\Delta\omega^{-1}}\uparrow_{\Delta G}^{G\times G},
\end{align*}
where $\cO_{\Delta\omega^{-1}}$ is the $\cO(\Delta G)$-module $\cO$ with the action of $\Delta G$ given by $(g,g).m=\omega(g)^{-1}m$, for all $m\in \cO$, so $M_{\omega^{-1}}$ certainly has linear source.
\end{proof}

\begin{thm}\label{thm:main2}
Let $p=2$ and $b$ a block with normal defect group $D$ such that $D/[D,D]$ has no direct factor isomorphic to $C_2$. Then $\Pic(b)=\cL(b)$.
\end{thm}

\begin{proof}
The proof is identical to that of Theorem~\ref{thm:main}. Assumption~\ref{ass}(ii) ensured that we were only considering blocks with abelian defect group that had no direct factor isomorphic to $C_2$ throughout $\S$\ref{sec:ext}. We can therefore apply Proposition~\ref{prop:good} in exactly the same way we did in the proof of Theorem~\ref{thm:main}.
\end{proof}

\begin{ack*}
The first author is supported by the EPSRC (grant no EP/T0\linebreak 04606/1). The authors would like to thank Dr. Charles Eaton for many useful conversations, particularly those that aided with the reduction to the abelian defect group case in the proof of Theorem \ref{thm:main}. We are also grateful to Prof. Peter Symonds for the guidance he gave us regarding group cohomology.
\end{ack*}

\end{document}